\documentclass{amsart}

\usepackage{amsmath}
\usepackage{amssymb}
\usepackage[]{amsrefs}
\usepackage{xpatch}
\usepackage{kantlipsum} 
\calclayout

\xpatchcmd{\proof}{\itshape}{\prooflabelfont}{}{}
\newcommand{\prooflabelfont}{\bfseries}

\DefineSimpleKey{bib}{primaryclass}{}
\DefineSimpleKey{bib}{archiveprefix}{}

\BibSpec{arXiv}{%
  +{}{\PrintAuthors}{author}
  +{,}{ \textit}{title}
  +{}{ \parenthesize}{date}
  +{,}{ arXiv }{eprint}
  +{,}{ primary class }{primaryclass}
}

\usepackage{kantlipsum} 
\calclayout
\usepackage{extarrows}
\usepackage{amssymb}
\usepackage[utf8]{inputenc}
\newtheorem{theorem}{Theorem}[section]
\newtheorem{proposition}[theorem]{Proposition}
\newtheorem{lemma}[theorem]{Lemma}
\newtheorem{corollary}[theorem]{Corollary}
\theoremstyle{definition}

\theoremstyle{definition}
\newtheorem{definition}[theorem]{Definition}
\newtheorem{remark}[theorem]{Remark}
\newtheorem{questions}[theorem]{Questions}

\DeclareMathOperator{\gr}{\text{gr}_{\mathfrak{m}}}
\DeclareMathOperator{\gll}{\text{g}\ell\ell}
\DeclareMathOperator{\ggl}{\bf{ggl}}

\numberwithin{equation}{section}
\usepackage{amsthm,amsmath,amssymb}
\usepackage[all,cmtip]{xy}
\usepackage{amsmath}
\usepackage{amssymb}
\usepackage{amsthm}
\usepackage[]{amsrefs}
\usepackage{hyperref}
\usepackage{xpatch}
\usepackage{amsfonts}
\usepackage{amssymb}
\usepackage[utf8]{inputenc}
\usepackage{amsthm}
\usepackage{stmaryrd}
\usepackage{csquotes}
\usepackage{extarrows}
\MakeOuterQuote{"}
\theoremstyle{definition}
\begin{document}
\date{\today}

\title[The index of a numerical semigroup ring]{The index of a numerical semigroup ring}

\author[Richard F. Bartels]{Richard F. Bartels}

\address{215 Carnegie Building, Dept. of Mathematics, Syracuse University, Syracuse, NY, 13244}

\email{rfbartel@syr.edu}

\urladdr{https://sites.google.com/view/richard-bartels-math/home}



\title[Generalized Loewy Length of Cohen-Macaulay Local and Graded Rings]{Generalized Loewy Length of Cohen-Macaulay Local and Graded Rings}
\begin{abstract}
We generalize a theorem of Ding relating the generalized Loewy length $\text{g}\ell\ell(R)$ and index of a one-dimensional Cohen-Macaulay local ring \,$(R,\mathfrak{m},k)$. Ding proved that if $R$ is Gorenstein, the associated graded ring is Cohen-Macaulay, and $k$ is infinite, then the generalized Loewy length and index of $R$ are equal. However, if $k$ is finite, equality may not hold. We prove that if the index of a one-dimensional Cohen-Macaulay local ring is finite and the associated graded ring has a homogeneous nonzerodivisor of degree $t$, then $\text{g}\ell\ell(R) \leq \text{index}(R)+t-1$. Next we prove that if $R$ is a one-dimensional hypersurface ring with a witness to the generalized Loewy length that induces a regular initial form on the associated graded ring, then the generalized Loewy length achieves this upper bound. We then compute the generalized Loewy lengths of several families of examples of one-dimensional hypersurface rings over finite fields. Finally, we study a graded version of the generalized Loewy length and determine its value for numerical semigroup rings.
\end{abstract}
\maketitle
\large{
\begin{center} \section{Introduction} \end{center} Let $(R,\mathfrak{m},k)$ be a Cohen-Macaulay local ring of Krull dimension $d$. For a finitely-generated $R$-module $M$, Auslander's $\delta$-invariant, denoted $\delta_{R}(M)$, is the smallest non-negative integer $n$ such that there exists a surjective $R$-map $X \oplus R^{n} \longrightarrow M$, where $X$ is a maximal Cohen-Macaulay $R$-module with no free direct summand. For finitely-generated $R$-modules $M$ and $N$, $\delta_{R}$ satisfies the following properties [11, Corollary 11.28]. \newline\newline
\text{}\,\,\,\,(i) $\delta_{R}(M \oplus N)=\delta_{R}(M)+\delta_{R}(N)$. \newline
\text{}\,\,(ii) $\delta_{R}(N) \leq \delta_{R}(M)$ if there is a surjective $R$-map $M \longrightarrow N$. \newline
\text{}(iii) $\delta_{R}(M) \leq \mu_{R}(M)$. \newline\newline
We see from properties (ii) and (iii) that for each $n \geq 1$,\newline \begin{equation*}0 \leq \delta_{R}(R/\mathfrak{m}^{n}) \leq \delta_{R}(R/\mathfrak{m}^{n+1}) \leq 1\,. \end{equation*} \text{}\newline Therefore, if $\delta_{R}(R/\mathfrak{m}^{n_{o}})=1$ for some $n_{o}$, then $\delta_{R}(R/\mathfrak{m}^{n})=1$ for all $n \geq n_{o}$. By a result of Auslander, regular local rings are precisely the Cohen-Macaulay local rings for which $\delta_{R}(R/\mathfrak{m}^{n})=1$ for all $n \geq 1$ [11, Proposition 11.37]. It is natural to ask when the sequence $\{\delta_{R}(R/\mathfrak{m}^{n})\}_{n=1}^{\infty}$ stabilizes at one for different classes of Cohen-Macaulay rings. The smallest positive integer $n$ for which $\delta_{R}(R/\mathfrak{m}^{n})=1$ is the following numerical invariant defined by Auslander. \newline
\[ \text{index}(R):=\text{inf}\{n \geq 1\,|\,\delta_{R}(R/\mathfrak{m}^{n})=1\}\]
\text{}\newline Suppose $R$ is a Cohen-Macaulay local ring with canonical module $\omega$. The {\it{trace}} of $\omega$ in $R$, denoted $\tau_{\omega}(R)$, is the ideal of $R$ generated by all $R$-homomorphic images of $\omega$ in $R$. Ding proved that if $R$ is a Cohen-Macaulay local ring with canonical module such that $\mathfrak{m} \subseteq \tau_{\omega}(R)$, then $\text{index}(R)$ is finite and bounded above by the {\it{generalized Loewy length}} of $R$ \,[6, Proposition 2.4]. This invariant, denoted $\gll(R)$, is the smallest positive integer $n$ for which $\mathfrak{m}^{n}$ is contained in the ideal generated by a system of parameters of $R$. In particular, $\text{index}(R) \leq \text{g}\ell\ell(R)$ if $R$ is Gorenstein. If, in addition to being Gorenstein, $R$ has infinite residue field and Cohen-Macaulay associated graded ring $\gr(R)$, then
$\text{index}(R)=\text{g}\ell\ell(R)$ \,[7, Theorem 2.1]. \text{}\newline\newline In general, if $R$ is a Cohen-Macaulay local ring that satisfies the above equality, we say that $R$ satisfies Ding's conjecture. In this paper, we study how the finiteness of the residue field can cause Ding's conjecture to fail. In particular, we prove that there are infinitely-many hypersurfaces with Cohen-Macaulay associated graded ring and finite residue field that do not satisfy Ding's conjecture. Each of our families of hypersurfaces generalizes an example of Hashimoto and Shida [8, Example 3.2], who showed for $R=\mathbb{F}_{2}\llbracket x,y \rrbracket/(xy(x+y))$ that $\text{index}(R)=3$\, and $\gll(R)=4$.\text{}\newline\newline 
When $k$ is finite, the assumption that $\gr(R)$ is Cohen-Macaulay does not guarantee the existence of a homogeneous system of parameters of degree one \,$x_{1}^{*},...,x_{d}^{*}$ \,in $(\gr(R))_{1}$. If a homogeneous system of parameters in $\gr(R)$ does not consist of linear elements, it cannot be used in Ding's argument to prove that $\text{index}(R)=\gll(R)$. \newline\newline However, if $R$ is a one-dimensional Cohen-Macaulay local ring with finite index and $\gr(R)$ is Cohen-Macaulay, then we can use a homogeneous $\gr(R)$-regular element of minimal degree to obtain an upper bound for $\gll(R)$ in terms of $\text{index}(R)$. In Theorem 2.3, we prove that if $R$ is one-dimensional Cohen-Macaulay and $\gr(R)$ has a homogeneous nonzerodivisor $z^{*}$, where $z \in \mathfrak{m}^{t} \setminus \mathfrak{m}^{t+1}$, then \newline
\[\gll(R) \leq \text{index}(R)+t-1\,. \]If $R$ is Gorenstein, then \[\text{index}(R) \leq \gll(R) \leq \text{index}(R)+t-1\,.\]\text{}\newline When $R$ is a hypersurface ring, we have $\text{index}(R)=e(R)$, where $e(R)$ denotes the Hilbert-Samuel multiplicity of $R$ [5, Theorem 3.3]. Therefore, the index of hypersurface rings is easy to compute: if $R=k\llbracket x_{1},...,x_{n}\rrbracket/(f)$, $\mathfrak{m}=(x_{1},...,x_{n})R$,\, and $f\in \mathfrak{m}^{r} \setminus \mathfrak{m}^{r+1}$, then\, $\text{index}(R)=e(R)=r$. \newline\newline In section 3, we prove that if $R$ is a one-dimensional hypersurface with a witness $z$ to its generalized Loewy length that induces a regular initial form on $\gr(R)$, then \newline \[ \gll(R)=\text{ord}_{R}(z)+e(R)-1.\] \text{}\newline
We then compute the generalized Loewy lengths of families of examples of one-dimensional hypersurface rings with finite residue field and Cohen-Macaulay associated graded ring. These examples illustrate differences between hypersurface rings $R$ with finite residue field and Cohen-Macaulay associated graded ring for which $\gll(R)=\text{index}(R)$ and $\gll(R)=\text{index}(R)+1$. In [3], De Stefani gave examples of one-dimensional Gorenstein local rings with infinite residue field for which $\gll(R)=\text{index}(R)+1$.
\newline\newline  
In section 4, we let $R$ be a positively-graded Noetherian $k$-algebra, where $k$ is an arbitrary field. We show that several families of one-dimensional standard graded hypersurfaces attain the graded version of the upper bound for the generalized Loewy length from Theorem 2.3. We then study a graded version of the generalized Loewy length: the {\it{generalized graded length}} of $R$, denoted $\ggl(R)$. After determining bounds for $\ggl(R)$ in terms of $\gll(R)$ and the minimum and maximum degrees of generators of $R$, we compute the generalized graded length of numerical semigroup rings. For $R=k[t^a,t^b]$, where $a<b$, we prove that $\ggl(R)=ba-b+1$ and if $z$ is a witness to $\ggl(R)$, then $(z)=(t^{ia})$ for some $1 \leq i \leq 1+b-a$.\newline\newline 
\section{\large{Estimating the Generalized Loewy Length of One-Dimensional Cohen-Macaulay Rings}}\text{}\newline Throughout this section, $(R,\mathfrak{m},k)$ is a local ring. We assume that $R$ \,has a nonzerodivisor $x$ of order $t$ such that multiplication by $x$ is injective on graded components of the associated graded ring in degrees less than $\text{index}(R)$. Generalizing [7, Lemma 2.3] to this context, we prove that if $R$ is a one-dimensional Cohen-Macaulay local ring with finite index, then $\gll(R) \leq \text{index}(R)+t-1$. \newline
\begin{lemma} Let $s$ and $t$ be positive integers and $x \in \mathfrak{m}^{t} \setminus \mathfrak{m}^{t+1}$ an $R$-regular element. Suppose the induced map $\overline{x}:\mathfrak{m}^{i-1}/\mathfrak{m}^{i} \longrightarrow \mathfrak{m}^{i+t-1}/\mathfrak{m}^{i+t}$ is injective for $1 \leq i \leq s$. Then \newline \[(\mathfrak{m}^{s+t-1},x)/x\mathfrak{m}^{s} \cong R/\mathfrak{m}^{s} \oplus (\mathfrak{m}^{s+t-1},x)/xR.\]\text{} \end{lemma}
\begin{proof} Let $I=xR \cap \mathfrak{m}^{s+t-1}$ and $W=(I+\mathfrak{m}^{s+t})/\mathfrak{m}^{s+t}$. Since $W$ is a $k$-subspace of $\mathfrak{m}^{s+t-1}/\mathfrak{m}^{s+t}$, there is a direct sum decomposition
\[\mathfrak{m}^{s+t-1}/\mathfrak{m}^{s+t}=W \oplus V\,\] for some subspace $V \subseteq \mathfrak{m}^{s+t-1}/\mathfrak{m}^{s+t}$. Let $e_{1},...,e_{n}$ be a $k$-basis for $V$. For each $i$, let $e_{i}=\overline{y_{i}}$, where $y_{i} \in \mathfrak{m}^{s+t-1}$. Let $B$ denote the $R$-submodule of $(\mathfrak{m}^{s+t-1},x)/x\mathfrak{m}^{s}$ generated by $[y_{1}],...,[y_{n}] \in (\mathfrak{m}^{s+t-1},x)/x\mathfrak{m}^{s}$. We will prove that $(\mathfrak{m}^{s+t-1},x)/x\mathfrak{m}^{s}=A \oplus B$, where $A=xR/x\mathfrak{m}^{s}$.
 First we show that 
 \[A+B=(\mathfrak{m}^{s+t-1},x)/x\mathfrak{m}^{s}. 
 \]\text{}\newline
 Choose $r_{1},...,r_{\alpha} \in R$ such that $I=(r_{1}x,...,r_{\alpha}x)$. Then $\mathfrak{m}^{s+t-1}/\mathfrak{m}^{s+t}$ is generated as a vector space by $\{\overline{r_{i}x}\}_{i=1}^{\alpha}\cup\{\overline{y}_{j}\}_{j=1}^{n}$, and by Nakayama's lemma, $\mathfrak{m}^{s+t-1}$ is generated as an $R$-module by $\{r_{i}x\}_{i=1}^{\alpha}\cup\{y_{j}\}_{j=1}^{n}$. Let $[z] \in (\mathfrak{m}^{s+t-1},x)/x\mathfrak{m}^{s}$. Then $[z]=r[x]+r'[v]$, where $r,r' \in R$, $v \in \mathfrak{m}^{s+t-1}$, and $v=r''x+\sum_{i=1}^{n}\rho_{i}y_{i}$, where $r'',\rho_{i} \in R$. So \newline
 \[[z]=(r+r'r'')[x]+\sum_{i=1}^{n}r'\rho_{i}[y_{i}] \in A+B.\]\text{}\newline
Now we show that $A \cap B=0$. Let $[z] \in A \cap B$. Then $[z]=a[x]=\sum_{i=1}^{n}a_{i}[y_{i}]$, where $a,a_{i} \in R$, and $ax-\sum_{i=1}^{n}a_{i}y_{i}\in x\mathfrak{m}^{s}$. Let $ax-\sum_{i=1}^{n}a_{i}y_{i}=xy$, where $y \in \mathfrak{m}^{s}$. Then $\sum_{i=1}^{n}a_{i}y_{i}=(a-y)x \in I$, so \newline
\[\overline{(a-y)x}=\overline{0} \in \mathfrak{m}^{s+t-1}/\mathfrak{m}^{s+t}.\]\text{}\newline If $a=y$ we are done, so assume $a-y \neq 0$. Then there is a nonnegative integer $l$ such that $a-y \in \mathfrak{m}^{l} \setminus \mathfrak{m}^{l+1}$. Suppose $0 \leq l<s$. Since $\overline{(a-y)x}=\overline{0}$ in $\mathfrak{m}^{l+t}/\mathfrak{m}^{l+t+1}$, it follows from the injectivity of the induced map $\overline{x}$ that $a-y \in \mathfrak{m}^{l+1}$, a contradiction. Therefore, $a-y \in \mathfrak{m}^{s}$, and $ax-xy \in x\mathfrak{m}^{s}$. Since $xy \in x\mathfrak{m}^{s}$, $ax \in x\mathfrak{m}^{s}$, and $[z]=a[x]=[0]$. \newline\newline It follows that $(\mathfrak{m}^{s+t-1},x)/x\mathfrak{m}^{s}=xR/x\mathfrak{m}^{s} \oplus B$ and $B \cong (\mathfrak{m}^{s+t-1},x)/xR$. Since $x$ is $R$-regular, it follows that $(\mathfrak{m}^{s+t-1},x)/x\mathfrak{m}^{s} \cong R/\mathfrak{m}^{s} \oplus (\mathfrak{m}^{s+t-1},x)/xR$.\end{proof}\text{}
\begin{lemma} Let $(R,\mathfrak{m})$ be a local ring, $I \subseteq R$ an ideal, and $x,y \in \mathfrak{m}$ such that $(x,I)=(y)$. If $I$ is not a principal ideal, then $(x)=(y)$. \end{lemma}
\begin{proof} Let $a,b \in R$ and $z \in I$ such that $y=ax+bz$. Let $c \in R$ such that $x=cy$. Then $y=acy+bz$ and $(1-ac)y=bz$. Suppose $c \in \mathfrak{m}$. Then $1-ac$ is invertible and $y=(1-ac)^{-1}bz \in I$, so $(y)=I$, which is false. Therefore $c$ is invertible and $(x)=(y)$. \end{proof}\text{}
\begin{theorem}\label{thm:ineq} Let $(R,\mathfrak{m})$ be a one-dimensional Cohen-Macaulay local ring for which $\text{index}(R)$ is finite. Let $s=\text{index}(R)$ and $x \in \mathfrak{m}^{t} \setminus \mathfrak{m}^{t+1}$ a nonzerodivisor, where $t \geq 1$. If the induced map \newline
\[\overline{x}:\mathfrak{m}^{i-1}/\mathfrak{m}^{i} \longrightarrow \mathfrak{m}^{i+t-1}/\mathfrak{m}^{i+t}\]
is injective for $1 \leq i \leq s$, then \[\gll(R) \leq \text{index}(R)+t-1.\] \text{} \newline If $\mathfrak{m}^{s+t-1}$ is not a principal ideal, then $\mathfrak{m}^{s+t-1} \subseteq (x).$
\end{theorem}\text{}
\begin{proof} By Lemma 2.1, $(\mathfrak{m}^{s+t-1},x)/x\mathfrak{m}^{s} \cong R/\mathfrak{m}^{s} \oplus (\mathfrak{m}^{s+t-1},x)/xR$, so there is a surjection 
\[
(\mathfrak{m}^{s+t-1},x) \longrightarrow R/\mathfrak{m}^{s}.
\]Therefore, $\delta_{R}((\mathfrak{m}^{s+t-1},x))>0$. By [10, Lemma 2.5], $(\mathfrak{m}^{s+t-1},x)$ is a parameter ideal of $R$. Let $(\mathfrak{m}^{s+t-1},x)=(y)$, where $y \in \mathfrak{m}$ is a regular element. Since $\mathfrak{m}^{s+t-1} \subseteq (y)$, we have $\gll(R) \leq s+t-1$. If $\mathfrak{m}^{s+t-1}$ is not a principal ideal, then by Lemma 2.2 we have $\mathfrak{m}^{s+t-1} \subseteq (x)$. \end{proof}\text{} 
\begin{definition}
Let $R$ be a Cohen-Macaulay local ring with canonical module $\omega$. The {\it{trace}} of $\omega$ in $R$, denoted $\tau_{\omega}(R)$, is the ideal of $R$ generated by all $R$-homomorphic images of $\omega$ in $R$. 
\end{definition}\text{} 
\begin{corollary} Let $(R,\mathfrak{m})$ be a one-dimensional Cohen-Macaulay local ring with canonical module $\omega$ such that $\mathfrak{m} \subseteq \tau_{\omega}(R)$. Let $x \in \mathfrak{m}^{t} \setminus \mathfrak{m}^{t+1}$ such that $x^{*} \in \gr(R)$ is a regular element. Then \[\text{index}(R) \leq \gll(R) \leq \text{index}(R)+t-1.\]\end{corollary}\begin{proof} This follows from [6, Proposition 2.4] and Theorem 2.3. \end{proof} \section{\large{Examples}} In this section we derive a formula for the generalized Loewy length of one-dimensional hypersurface rings and compute the generalized Loewy lengths of several families of examples of one-dimensional hypersurfaces. The associated graded ring of each of these hypersurface rings has a homogeneous nonzerodivisor of degree one or two, so the index and generalized Loewy length differ by at most one. \newline\newline Using techniques from the proof of [8, Example 3.2], we prove that for several families of hypersurfaces $\{R_{n}\}_{n=1}^{\infty}$,  
\[\gll(R_{n})-\text{index}(R_{n})=1\]\text{}\newline for $n \geq 1$. This difference is positive for each $n$ because of the absence of a regular linear form in certain one-dimensional hypersurface rings over finite fields.
\newline\newline 
Throughout this section, $S=k\llbracket x,y \rrbracket$, where $k$ is a field and $\mathfrak{n}=(x,y)S$. We say that the \textit{order} of an element $f \in S$ is $r$ if $f \in \mathfrak{n}^{r} \setminus \mathfrak{n}^{r+1}$, and write $\text{ord}_{S}(f)=r$. Let $R=S/(f)$, where $f \in \mathfrak{n}$. Let $\mathfrak{m}=(x,y)R$. The order of an element $z \in R$ is $r$ if $z \in \mathfrak{m}^{r} \setminus \mathfrak{m}^{r+1}$, and we write $\text{ord}_{R}(z)=r$. Recall that $\text{index}(R)=e(R)$. Finally, if $(R,\mathfrak{m})$ is any local ring of embedding dimension $n$, then $\mu_{R}(\mathfrak{m}^{r}) \leq \binom{n+r-1}{r}$.\newline\newline 
\begin{lemma} Let $R=k \llbracket x,y \rrbracket/(f)$, where $\text{ord}_{S}(f)=e$ and $g=\gll(R)$. Let $z \in \mathfrak{m}$ such that $\mathfrak{m}^{g} \subseteq (z)$ and $i \geq 0$. If $\gll(R) \leq e+i$, then $\text{ord}_{R}(z) \leq i+1$.
\end{lemma} 
\begin{proof} Let $\text{ord}_{R}(z)=r$ and $\zeta \in \mathfrak{n}^{r} \setminus \mathfrak{n}^{r+1}$ such that $\overline{\zeta}=z$. Then $\mathfrak{n}^{g} \subseteq (f, \zeta)$. Let $M$ be the $k$-vector space of leading forms of degree $g$ of elements of $(f,\zeta)$. Since $\text{ord}_{S}(\zeta)=r$, we obtain leading forms of degree $g$ from this element by multiplying $\zeta$ by generators of $\mathfrak{n}^{g-r}$. Similarly, we multiply $f$ by generators of $\mathfrak{n}^{g-e}$ to obtain leading forms of degree $g$. Therefore,\newline 
\[\text{dim}_{k}\,M \leq \binom{2+(g-e)-1}{g-e}+\binom{2+(g-r)-1}{g-r}=2g-(e+r)+2. 
\]\text{} \newline
On the other hand, the vector space of forms of degree $g$ in $\mathfrak{n}^{g}$ has dimension $g+1$. Therefore, $g+1 \leq 2g-(e+r)+2$ \,\,and\,\, $e+r \leq g+1$. The result follows from this inequality. \end{proof}\text{}
\begin{definition}
Let $(R,\mathfrak{m})$ be a $d$-dimensional local ring. A system of parameters\, $\underline{\bf{x}}=x_1,...,x_d \in \mathfrak{m}$ is a {\it{witness to}} $g=\gll(R)$ if $\mathfrak{m}^g \subseteq (\underline{\bf{x}})$.
\end{definition}
If $R$ is a one-dimensional hypersurface with a witness $z$ to $\gll(R)$ that induces a regular initial form on $\gr(R)$, then we can compute $\gll(R)$ using the following formula. We see that the order of $z$ is uniquely determined by $\gll(R)$ and $e(R)$.
\begin{proposition}
 Let $R=k \llbracket x,y \rrbracket/(f)$, where $\text{ord}_{S}(f)=e$ and $z \in \mathfrak{m}$ such that $z^{*}$ is $\gr(R)$-regular. If $z$ is a witness to $\gll(R)$, then  \[\gll(R)=\text{ord}_{R} (z)+e-1.\]
\end{proposition}
\begin{proof} Recall that $\text{index}(R)=e$. Let $g=\gll(R)$ and $n=g-e$. Then \,$g=e+n$\, and by Lemma 3.1, $\text{ord}_{R}(z) \leq n+1$. By Theorem 2.3,\, $g \leq e+\text{ord}_{R}(z)-1 \leq e+n=g$.
\end{proof}\text{}\newline 
If we cannot find an element of a one-dimensional hypersurface that is a witness to $\gll(R)$ and induces a regular form on $\gr(R)$, then we can use the following lemma to estimate the generalized Loewy length. \newline  
\begin{lemma} Let $R=k \llbracket x,y \rrbracket/(f)$, where \,$\text{ord}_{S}(f)=e \geq 2$. If $R$ has no nonzerodivisors of the form $\alpha x+\beta y$, where $\alpha,\beta \in k$, then $\gll(R)>e.$\text{}\end{lemma}\begin{proof} Since $\text{index}(R)=e$, we have $e \leq \gll(R)$ [6, Proposition 2.4]. Suppose $\gll(R)=e$. Let $z \in \mathfrak{m}$ such that $\mathfrak{m}^{e} \subseteq (z)$. By Lemma 3.1, we have $\text{ord}_{R}(z)=1$. Let $\zeta \in \mathfrak{n} \setminus \mathfrak{n}^2$ be a preimage of $z$. Note that for each invertible matrix $\begin{pmatrix}
a & b \\
c & d 
\end{pmatrix} \in \text{GL}_2(k)$, the map $\phi:S \longrightarrow S$ defined by $\phi(x)=ax+by$ and $\phi(y)=cx+dy$ is a $k$-algebra automorphism. Letting an appropriate invertible matrix in $\text{GL}_{2}(k)$ act on $S$, we may assume without loss of generality that $\zeta=x-h$ for some nonzero element $h \in S$ with $\text{ord}_{S}(h) \geq 2$. Since $x$ is a zerodivisor on $R$, there is an element $g \in \mathfrak{n}^{e-1}$ such that $f=xg$. \newline\newline
Let $R'=S/(\zeta)$. Since $S$ is a regular local ring and $\text{ord}_{S}(\zeta)=1$, it follows that $R'$ is a one-dimensional regular local ring, and thus a discrete valuation ring. Let $\overline{f}$ denote the image of $f$ in $R'$. Then 
\[R/(z) \cong R'/(\overline{f}).\]\text{} \newline
Since $\overline{g} \in (x,y)^{e-1}R'$ and $\overline{x}=\overline{h} \in (x,y)^2R'$, it follows that $\overline{f} \in (x,y)^{e+1}R'$, so
$l_{R'}(R'/(\overline{f}))=\text{ord}_{R'}(\overline{f}) \geq e+1$  and  $l_{R}(R/(z)) \geq e+1$. Now let $R_{1}:=R/(z)$\, and $\mathfrak{m}_{1}:=\mathfrak{m}/(z)$.
Then \newline  
\[0=\mathfrak{m}_{1}^{e} \subseteq \mathfrak{m}_{1}^{e-1} \subseteq \cdot\cdot\cdot \subseteq \mathfrak{m}_{1} \subseteq R_{1}\]\text{}\newline is a composition series for $R_{1}$, so\, $l_{R}(R/(z))=e$. This is a contradiction.\end{proof} \text{}\newline 
If $(R,\mathfrak{m},k)$ is a one-dimensional local ring with Cohen-Macaulay associated graded ring and infinite residue field, then $\gr(R)$ has a homogeneous linear nonzerodivisor. We now consider one-dimensional hypersurface rings with finite residue field such that the associated graded ring does not have a homogeneous linear nonzerodivisor. If the associated graded ring has a homogeneous quadratic nonzerodivisor, then it follows from Theorem 2.3 and Lemma 3.4 that the difference between the generalized Loewy length and index is one. \newline 
\begin{proposition} Let $k$ be a finite field and $R=k\llbracket x,y \rrbracket/y(\prod\limits_{\alpha \in k}(x+\alpha y))$. Then \[\gll(R)=\text{index}(R)+1=|k|+2.\] \end{proposition}
\begin{proof} We construct a homogeneous nonzerodivisor of degree $2$ in $\gr(R)$. Let $f \in k[x]$ be a degree 2 irreducible polynomial. Define
\[g(x,y) \in k[x,y] \,\,\,by\,\,\, g(x,y):=y^2f(\frac{x}{y}).\]\text{}\newline
We claim that the element $\overline{g}=g(\overline{x},\overline{y}) \in \gr(R)=k[x,y]/y(\prod\limits_{\alpha \in k}(x+\alpha y))$ is $\gr(R)$-regular. \newline\newline Let $h \in k[x,y]$ such that $\overline{g}\overline{h}=\overline{0}$. Then there exists a polynomial $p(x,y) \in k[x,y]$ such that \newline
\[
gh=py(\prod\limits_{\alpha \in k}(x+\alpha y)). 
\]\text{}\newline
Let $\alpha \in k$. Suppose $(x+\alpha y) \mid g$ and $q(x,y) \in k[x,y]$ such that $(x+\alpha y)q(x,y)=g(x,y)$. Then $(x+\alpha)q(x,1)=g(x,1)=f(x)$. This contradicts the irreducibility of $f$. \newline\newline It follows that $(x+\alpha y) \mid h$. Clearly $y \nmid g$, so $y \mid h$ as well, and $y(\prod\limits_{\alpha \in k}(x+\alpha y)) \mid h$. Therefore we have $\overline{h}=\overline{0}$, and $\overline{g}$ is $\gr(R)$-regular. By Theorem 2.3 and Lemma 3.4, $\gll(R) = \text{index}(R)+1$. \end{proof} \text{}\newline
\begin{remark} When $k=\mathbb{F}_{2}$, Proposition 3.5 is Hashimoto and Shida's counterexample to Ding's conjecture: $\mathbb{F}_{2}\llbracket x,y \rrbracket/(xy(x+y))$. In the following propositions, we compute the generalized Loewy lengths of families of one-dimensional hypersurface rings of the form $k\llbracket x,y \rrbracket/(xy(x^{n}+y^{n}))$, where $k$ is a finite field and $n$ is a positive integer.\end{remark}\text{}
\begin{proposition} Let $n \geq 1$ and $k$ a field such that $\text{char}\,k \neq 2$ and $\text{char}\,k \nmid 1+(-2)^{n}$. Let  $R=k\llbracket x,y\rrbracket/(xy(x^n+y^n))$. Then $\mathfrak{m}^{n+2}=(x+2y)\mathfrak{m}^{n+1}$ and 
\[\gll(R)=\text{index}(R)=n+2.
\] \end{proposition}
\begin{proof} Since $\mathfrak{m}^{n+1}$ is generated by $\{x^{n+1-i}y^{i} \}_{i=0}^{n+1}$, it follows that
$(x+2y)\mathfrak{m}^{n+1}$ is generated by $\{x^{n+2-i}y^{i}+2x^{n+1-i}y^{i+1}\}_{i=0}^{n+1}$.
Let \[z_{i}=x^{n+2-i}y^{i}+2x^{n+1-i}y^{i+1}\]\newline for $0 \leq i \leq n+1$. Since $xy^{n+1}=-x^{n+1}y$, \newline
\[\sum_{i=1}^{n} (-2)^{i-1}z_{i}=x^{n+1}y+2(-2)^{n-1}xy^{n+1}
\]
\[=x^{n+1}y-2(-2)^{n-1}x^{n+1}y
\] 
\[=(1+(-2)^{n})x^{n+1}y.
\]\text{} \newline Since $z_i \in (x+2y)\mathfrak{m}^{n+1}$ for\, $0 \leq i \leq n+1$, we have\, $x^{n+1}y \in (x+2y)\mathfrak{m}^{n+1}$ and $\mathfrak{m}^{n+2} \subseteq (x+2y)\mathfrak{m}^{n+1}$. Therefore, $\gll(R) \leq n+2=\text{index}(R) \leq \gll(R)$. \end{proof}\text{}
\begin{corollary} Let $k$ be a field of characteristic $p>2$ and $R=k\llbracket x,y \rrbracket /(xy(x^{p^{n}}+y^{p^{n}}))$, where $n \geq 0$. Then $\mathfrak{m}^{p^{n}+2}=(x+2y)\mathfrak{m}^{p^{n}+1}$, and 
\[\gll(R)=\text{index}(R)=p^{n}+2.
\] \end{corollary}
\begin{proof} Suppose $p \mid 1+(-2)^{p^{n}}$. Since $1+(-2)^{p^{n}}=1-2^{p^{n}}$, we have $2^{p^{n}}=1 \text{ mod }p$. Since $2^{p^{n}}=2 \text{ mod }p$, it follows that $2=1 \text{ mod }p$, which is false. Therefore, $p \nmid 1+(-2)^{p^{n}}$. The result now follows from Proposition 3.7.\end{proof} \text{}\newline 
If we let $p=2$ in Corollary 3.8, then the generalized Loewy length and index of $R$ differ by one. This is a special case of Proposition 3.12. To prove Proposition 3.12, we require the following results about the reducibility of cyclotomic polynomials modulo prime integers and primitive roots of powers of prime integers. \newline
\begin{lemma}[12, Theorem 2.47] Let $K=\mathbb{F}_{q}$, where $q$ is prime and $q \nmid n$. Let $\varphi$ denote Euler's totient function and $d$ the least positive integer such that $q^{d}=1 \text{ mod } n$. Then the $n^{\text{th}}$ cyclotomic polynomial \,$\Phi_{n}$ factors into $\varphi(n)/d$ distinct monic irreducible polynomials in $K[x]$ of degree $d$.  
\end{lemma} \text{} \begin{lemma}[12, Example 2.46] Let $p$ be prime and $m \in \mathbb{N}$. Then the $p^{m}$th cyclotomic polynomial $\Phi_{p^{m}}$ equals \[1+x^{p^{m-1}}+x^{2p^{m-1}}+\,\cdot \cdot \cdot\, +x^{(p-1)p^{m-1}}. 
\]\end{lemma}\text{}
\begin{lemma}[2, Proposition 3.4.1] Let $p$ be a prime and $g$ a positive integer. Then the following three assertions are equivalent: \begin{enumerate} 
\item $g$ is a primitive root modulo $p$ and $g^{p-1} \neq 1 \text{ mod } p$;  
\item $g$ is a primitive root modulo $p^2$; 
\item For every $i \geq 2$, $g$ is a primitive root modulo $p^{i}$.\end{enumerate} \end{lemma} \text\newline
\begin{proposition} Let $R=\mathbb{F}_{2}\llbracket x,y \rrbracket/(xy(x^{2^{n}p^{m}}+y^{2^{n}p^{m}}))$,
where $m,n \geq 0$ and $p>3$ is a prime such that $2$ is a primitive root modulo $p^2$. Then \newline
\[\gll(R)=\text{index}(R)+1=2^{n}p^{m}+3
\]\text{}
and \[\mathfrak{m}^{2^{n}p^{m}+3} \subseteq (x^2+xy+y^2).
\]\text{}\newline If $m=1$, then we need only assume that $2$ is a primitive root modulo $p$. \end{proposition}\begin{proof} First assume that $m>0$. We show that $x^2+xy+y^2$ is $\gr(R)$-regular. Let $S=\mathbb{F}_{2}\llbracket x,y \rrbracket$ and suppose $f,g \in S$ such that \newline \begin{equation} (x^2+xy+y^2)f=g(xy(x^{2^{n}p^{m}}+y^{2^{n}p^{m}}))=g(xy(x^{p^{m}}+y^{p^{m}})^{2^{n}}). \end{equation}\text{}\newline By Lemmas 3.9 through 3.11, $\Phi_{p^{i}}(x)$ is an irreducible polynomial over $\mathbb{F}_{2}$ of degree $p^{i}-p^{i-1}$ for $1 \leq i \leq m$. We obtain the following factorization of $x^{p^{m}}+1$ into irreducible polynomials over $\mathbb{F}_{2}$. 
\[x^{p^{m}}+1=(x+1)\prod\limits_{i=1}^{m}\Phi_{p^{i}}(x).
\] Let $h_{i}(x,y):=y^{p^{i}-p^{i-1}}\Phi_{p^{i}}(x/y)$ for $i=1,...,m$. Then $h_{i}(x,y)$ is a homogeneous polynomial of degree $p^{i}-p^{i-1}$, and \begin{equation} x^{p^{m}}+y^{p^{m}}=(x+y)\prod\limits_{i=1}^{m}h_{i}(x,y).\end{equation}\text{} We claim that each $h_{i}(x,y)$ is irreducible over $\mathbb{F}_{2}$. Suppose $p,q \in \mathbb{F}_{2}[x,y]$ such that 
\[
h_{i}(x,y)=p(x,y)q(x,y). 
\] 
Since $h_{i}$ is homogeneous, $p$ and $q$ are homogeneous. Let $y=1$ in the above equation. Then  
\[
\Phi_{p^{i}}(x)=h_{i}(x,1)=p(x,1)q(x,1). 
\]
Since $\Phi_{p^{i}}$ is irreducible over $\mathbb{F}_{2}$, $p(x,1)=\Phi_{p^{i}}(x)$ or $q(x,1)=\Phi_{p^{i}}(x)$. Assume $p(x,1)=\Phi_{p^{i}}(x)$. Then $p(x,y)=h_{i}(x,y)$, so $h_{i}(x,y)$ is irreducible. By equations (3.1) and (3.2), $h_{i} \mid (x^2+xy+y^2)$ or $h_{i} \mid f$. Since the degree of $h_{i}$ is 
\[
p^{i}-p^{i-1}=p^{i-1}(p-1) \geq p^{i-1}3, 
\]
it follows that $h_{i} \mid f$. It is clear that $x$, $y$, and $x+y$ divide $f$ as well, so $f \in (xy(x^{p^{m}}+y^{p^{m}})^{2^{n}})$, and $x^2+xy+y^2$ is a nonzerodivisor on $\gr(R)$. By Theorem 2.3 and Lemma 3.4, $\gll(R) = \text{index}(R)+1$ and $\mathfrak{m}^{2^{n}p^{m}+3} \subseteq (x^2+xy+y^2)$. If $m=0$, \,(3.1) becomes 
\[(x^2+xy+y^2)f=g(xy(x^{2^{n}}+y^{2^{n}}))=g(xy(x+y)^{2^{n}}).
\]\text{}It follows that $f \in xy(x^{2^{n}}+y^{2^{n}})$, so $x^2+xy+y^2$ is a nonzerodivisor on $\gr(R)$. Therefore, $\mathfrak{m}^{2^{n}+3} \subseteq (x^2+xy+y^2)$ and $\gll(R) = \text{index}(R)+1$.  \end{proof}\text{}
\begin{remark} Whether there are infinitely many primes $p$ such that $2$ is a primitive root modulo $p$ is an open question. This is a special case of Artin's conjecture on primitive roots [2, p.66]. A list of the first primes $p$ for which $2$ is a primitive root modulo $p$ is sequence A001122 in the OEIS.
\end{remark}\text{} 
\section{\large{Generalized Loewy Length of Graded Algebras}}  
We now consider positively-graded Noetherian $k$-algebras and a graded analogue of the generalized Loewy length of a local ring. Throughout this section, $k$ is an arbitrary field. \newline
\begin{definition}
Let $R=\bigoplus\limits_{i \geq 0}R_{i}$ be a positively-graded Noetherian $k$-algebra, where $R_{0}=k$ and $\mathfrak{m}=\bigoplus\limits_{i \geq 1}R_{i}$ is the irrelevant ideal. For $n \geq 0$, let $\mathfrak{m}_{n}:=\bigoplus\limits_{i \geq n}R_{i}$. The {\it{generalized graded length}} of $R$, denoted $\ggl(R)$, is the smallest positive integer $n$ for which $\mathfrak{m}_{n}$ is contained in the ideal generated by a homogeneous system of parameters.\end{definition}\text{}\newline
In this context, the generalized Loewy length, $\gll(R)$, is the smallest positive integer $n$ for which $\mathfrak{m}^{n}$ is contained in the ideal generated by a homogeneous system of parameters.\newline\newline
With Herzog, we note that all of the above definitions can be transferred accordingly to standard graded Gorenstein $k$-algebras [9, p.98]. For $R=k[x,y]/(f)$, one can prove that $\text{index}(R)=\text{deg}(f)$ by using Ding's arguments in [5] to prove the standard graded version of [5, Theorem 3.3]. By the standard graded version of Proposition 3.3, the generalized Loewy length of $k[x,y]/(f)$ is one less than the sum of the degree of $f$ and the degree of a witness to $\text{g}\ell\ell(R)$.\newline 
\begin{proposition}
Let $R=k[x,y]/(f)$ be standard graded, where $f\in k[x,y]$ is a form of degree $e$. Let $z \in (x,y)R$ be a witness to $\gll(R)$. Then $\gll(R)=\text{deg}_{R}(z)+e-1$.\end{proposition}\text{}
\newline 
Let $(R,\mathfrak{m})$ be a positively-graded Noetherian $k$-algebra. It is clear that for each $n \geq 1$, we have \,$\mathfrak{m}^{n} \subseteq \mathfrak{m}_{n}$,\, so\, $\gll(R) \leq \ggl(R)$. We now determine upper and lower bounds for $\ggl(R)$ in terms of $\gll(R)$ and the minimum and maximum degrees of generators of $R$.\newline
\begin{proposition} Let $(R,\mathfrak{m})$ be a positively-graded Noetherian $k$-algebra, where $R_{0}=k$ and $\mathfrak{m}$ is the irrelevant ideal. Suppose $x_{1},...,x_{n} \in \mathfrak{m}$ are homogeneous elements such that $R=k[x_{1},...,x_{n}]$. Let 
\[
\text{min}\{\text{deg}(x_{i})\}_{i=1}^{n}=a \leq b = \text{max}\{\text{deg}(x_{i})\}_{i=1}^{n}.
\]
Then 
\[
a(\gll(R))-(a-1)^2 \leq  \ggl(R) \leq b(\gll(R))-b+1.
\]
\text{}\newline  
If $a=b=1$, then $\ggl(R)=\gll(R)$.
\end{proposition} 
\begin{proof} We claim that for $n \geq 0$, 
$\mathfrak{m}_{nb+1} \subseteq \mathfrak{m}^{n+1}$. 
This is trivial when $n=0$. Suppose the inclusion holds for some $n \geq 0$. Let $x \in \mathfrak{m}_{(n+1)b+1}$ be homogeneous, and suppose $x=\sum\limits_{i=1}^{n}s_{i}x_{i}$, where each $s_{i} \in R$ is homogeneous. Then $\text{deg}(s_{i}) \geq (n+1)b+1-\text{deg}(x_{i}) \geq nb+1.$ Therefore, $s_{i} \in \mathfrak{m}_{nb+1} \subseteq \mathfrak{m}^{n+1}$, and $x \in \mathfrak{m}^{n+2}$. This proves the claim. Let $n=\gll(R)-1$. Then by the above inclusion, $\ggl(R) \leq b(\gll(R)-1)+1$. \newline\newline 
Let $m=\ggl(R)$. There exists an integer $c\geq 0$ and an integer $0 \leq l < a$ such that $m=ac+l$. It is clear that $\mathfrak{m}^{i} \subseteq \mathfrak{m}_{ia}$ for $i \geq 0$. We claim that $\mathfrak{m}^{i+j} \subseteq \mathfrak{m}_{ia+j}$ for $i,j \geq 0$. Fix $i$. If the inclusion holds for some $j \geq 0$, then 
\[
\mathfrak{m}^{i+j+1}=\mathfrak{m} \cdot \mathfrak{m}^{i+j} \subseteq \mathfrak{m} \cdot \mathfrak{m}_{ia+j} \subseteq \mathfrak{m}_{ia+j+1}.
\] 
It follows that $\mathfrak{m}^{c+l} \subseteq \mathfrak{m}_{m}$ and $c+l \geq n=\gll(R)$. Since $ac+al \geq an$, we have 
\[
\ggl(R) \geq a(\gll(R))-(a-1)l
\]
and
\[
\ggl(R) \geq a(\gll(R))-(a-1)^2.
\]\end{proof}\text{}\newline
Let $H=\langle a_1,...,a_n \rangle$ be the numerical semigroup with unique minimal generating set $0<a_1<a_2<\cdot\cdot\cdot <a_n$, where $\text{gcd}(a_1,...,a_n)=1$. Let $C$ denote the {\it{conductor}} of $H$, the smallest integer $n \in H$ for which every integer larger than $n$ is also in $H$. Define $k[H]:=k[t^{a_1},...,t^{a_n}] \subseteq k[t]$, where $k[H]$ is positively-graded via $|t^a|=a$.
\newline
\begin{proposition} Let $R=k[H]$, where $H=\langle a_1,...,a_n \rangle$. Then \,$\ggl(R) = C+a_{1}$.\text{} \end{proposition}
\begin{proof} 
Let $\mathfrak{m}=(t^{a_{1}},...,t^{a_{n}})$. It is clear that $\mathfrak{m}_{C+a_{1}} \subseteq (t^{a_{1}})$, so \,$\ggl(R) \leq C+a_{1}$. Let $n,d \geq 0$ and suppose $\mathfrak{m}_{C+n} \subseteq (t^d)$. This inclusion holds if and only if $t^{C+n+i} \in (t^d)$ for all $i \geq 0$, which is true if and only if\,
$C+n+i-d \in H$ for all $i \geq 0$. This is equivalent to the inequality
$C+n-d \geq C$, or
$n \geq d$.
Therefore, $\mathfrak{m}_{C+a_{1}-1} \not\subseteq (t^d)$ for all $d \in H \setminus \{0\}$. It follows that $\ggl(R)=C+a_{1}$.
\end{proof}\text{}
\begin{corollary}
Let $R=k[t^a,t^b]$, where $a<b$. Then \,$\ggl(R)=ba-b+1$.
\end{corollary}
\begin{proof}
The conductor of \,$\langle a,b \rangle$\, is \,$ba-a-b+1$ \,[14, p.201].
\end{proof}\text{}\newline 
Veliche notes that for $R=k\llbracket t^a,t^b \rrbracket$, where $a<b$\ and $k$ is infinite, we have $\gll(R)=\text{index}(R)=a$\, [16, p.3]. She then determines formulas for the generalized Loewy lengths of Gorenstein local numerical semigroup rings of embedding dimension at least three over infinite fields. [16, Corollary 2.4, Corollary 3.3, Proposition 3.9]. If we know the conductor of the semigroup that determines one of these rings, then the generalized graded length of the corresponding graded ring is easier to compute than the generalized Loewy length of this local ring.
\newline
\begin{proposition} Let $R=k[H]$, where $H=\langle a,b \rangle$ and $a<b.$ Suppose $z$ is a witness to $\ggl(R)$. Then $(z)=(t^{ia})$ for some $1 \leq i \leq 1+b-a$.
\end{proposition} 
\begin{proof} We have $R \cong k[x,y]/(x^b-y^a)=k[\overline{x},\overline{y}]$, where $\text{deg}(\overline{x})=a$, $\text{deg}(\overline{y})=b$, and $\mathfrak{m}=(\overline{x},\overline{y})$. Let $z \in \mathfrak{m}$ be a witness to $\ggl(R)$. Suppose $z \in (\overline{y})$. By Proposition 4.4, $\overline{x}$ is also a witness to $\ggl(R)$, so \,$\mathfrak{m}_{ab-(b-1)} \subseteq (\overline{x}) \cap (\overline{y})$. Since $a<b$ are coprime, we have $b=as+r$ for some $s>0$ and $0<r<b$, so \newline
\[
ab-(b-1)=a(as+r)-(as+r-1)=a((a-1)s+r)-(r-1).
\]\text{}\newline It follows that $\overline{x}^{(a-1)s+r} \in (\overline{x})\, \cap\, (\overline{y})=(\overline{x}\overline{y}, \overline{x}^{b})$. Since $(a-1)s+r<b$, we have $\overline{x}^{(a-1)s+r}=\overline{f}\overline{x}\overline{y}$ for some $f \in k[x,y]$ and $x^{(a-1)s+r}-fxy=gx^{b}-gy^{a}$ for some $g \in k[x,y]$. It follows that 
$fxy-gy^a=x^{(a-1)s+r}-gx^b$. 
Therefore, $y\,|\,(x^{(a-1)s+r}-gx^{b})$. Since $(a-1)s+r<b$, this is false. We conclude that $z \not\in (\overline{y})$. \newline\newline Therefore, $z=\overline{x}^{i}$ for some $1 \leq i<b$. For each $n \geq 1$, we have $\mathfrak{m}^{n} \subseteq \mathfrak{m}_{na}$, so  
\[
\mathfrak{m}^{b} \subseteq \mathfrak{m}_{ab} \subseteq \mathfrak{m}_{ab-(b-1)}.
\]
It follows that $\mathfrak{n}^{b} \subseteq (x^i,x^b-y^a)S=(x^i,y^a)S$, where $\mathfrak{n}=(x,y)S$, and $S$ denotes $k[x,y]$ with the standard grading. Let $M$ denote the $k$-vector space generated by monomials of degree $b$ in $(x^i, y^a)S$. Then 
\[\text{dim}_{k}M \leq \binom{2+(b-i)-1}{b-i}+\binom{2+(b-a)-1}{b-a}=2+2b-(a+i).
\] On the other hand, the $k$-vector space generated by monomials of degree $b$ in $\mathfrak{n}^{b}$ has dimension $b+1$. It follows that $b+1 \leq 2+2b-(a+i)$ and $i \leq 1+b-a$.
\end{proof}\text{} 
\begin{definition}
Let $(R,\mathfrak{m})$ be a positively-graded Noetherian $k$-algebra, where $R_{0}=k$ and $\mathfrak{m}$ is the irrelevant ideal. Let $I \subseteq R$ be a graded ideal. We say that $I$ is a {\it{graded reduction}} of $\mathfrak{m}$ of degree $d$ if there is a positive integer $i$ such that $I\mathfrak{m}_{i}=\mathfrak{m}_{i+d}$. 
\end{definition}\text{}\newline 
It is clear that for a numerical semigroup ring $k[t^{a_{1}},...,t^{a_{n}}]$, the ideal $(t^{a_{1}})$ is a graded reduction of $\mathfrak{m}$ of degree $a_{1}$. We therefore ask the following questions, which parallel a question asked by De Stefani [3, Questions 4.5 (ii)]. \newline
\begin{questions}
Suppose $R$ is a positively-graded Noetherian $k$-algebra. Is there a witness to $\ggl(R)$ that generates a graded reduction of $\mathfrak{m}$? What can be said about the degree of such a graded reduction?  
\end{questions}

\section*{Acknowledgements}
I would like to thank my thesis advisor Graham Leuschke for his support and insight into this topic, and for many helpful conversations. I would like to thank Elo\'isa Grifo for her helpful suggestions for the abstract of this paper. Finally, I would like to thank the referee for providing helpful comments and suggestions for this paper.

\bibliographystyle{amsplain}
}
\end{document}